\documentclass[a4paper,11pt]{article}
\usepackage{bbm}
\usepackage{amsfonts}
\usepackage{amsthm,amsmath}
\usepackage[integrals]{wasysym}
\usepackage{mathrsfs}
\usepackage{amsmath,multicol,amsopn}
\usepackage{latexsym,amssymb}
\usepackage{amsbsy, amstext}
\usepackage{graphicx}    
\usepackage{indentfirst} 
\usepackage{graphicx}
\usepackage{subfigure}
\usepackage{color}
\def\be{\begin{array}}
\def\en{\end{array}}

\numberwithin{equation}{section}

\newtheorem{theorem}{Theorem}[section] 
\newtheorem{definition}{Definition}[section]
\newtheorem{coro}{Corollary}[section]
\newtheorem{lemma}{Lemma}[section]
\newtheorem{remark}{Remark}[section]
\newtheorem{proposition}{Proposition}[section]

\newcommand{\R}{{\mathbb R}}
\newcommand{\pozhehao}{\kern0.3ex\rule[0.8ex]{1.5em}{0.095ex}\kern0.3ex}

\begin{document}
\title{Global well-posedness, stability and instability for the non-viscous Oldroyd-B model}
\author{
Weikui $\mbox{Ye}^{1}$\footnote{email: 904817751@qq.com}\\
 $^1\mbox{Institute}$ of Applied Physics and Computational Mathematics,\\
P.O. Box 8009, Beijing 100088, P. R. China\\}
\date{}
\maketitle

\begin{center}
\begin{minipage}{15.5cm}

{\bf Abstract.} In this paper we consider the 3-dimensional incompressible Oldroyd-B model.
First, we establish two results of the global existence for different kinds of the coupling coefficient $k$. Then, we prove that the solutions $(u,\tau)$ are globally steady when $k^m\rightarrow k>0$, though $(u,\tau)$ corresponds to different decays for different kinds of $k>0~$. Finally, we show that the energy of $u(t,x)$ will have a jump when $k\rightarrow 0$ in large time, which implies a non-steady phenomenon. In a word, we find an interesting physical phenomenon of \eqref{1} such that smaller coupling coefficient $k$ will have a better impact for the energy dissipation of $(u,\tau)$, but $k$ can't be too small to zero, or the dissipation will vanish instantly. While the damping term $\tau$ and $\mathbb{D}u$ always bring the well impact for the energy dissipation.


\end{minipage}
\end{center}
\vskip 6mm

\section{Introduction and main results}
 \setcounter{equation}{0}
\par
In this paper, we study the incompressible Oldroyd-B model of the non-Newtonian fluid in $\R^+\times\R^d$
\begin{equation}\label{gs}
\begin{cases}
\partial_tu+(u\cdot\nabla)u-\nu \Delta u+\nabla p=k\mathrm{div}(\tau),\\
\partial_t\tau+(u\cdot\nabla)\tau-\eta\Delta\tau+\mu\tau+\mathrm{Q}(\nabla u,\tau)=\alpha\mathbb{D}u,\\
\mathrm{div} u=0,\\
u(x,0)=u_0(x),~~~\tau(0,x)=\tau_0(x),
\end{cases}
\end{equation}
where $u$ denotes the velocity, $\tau=\tau_{i,j}$ is the non-Newtonian part of the stress tensor($\tau$ is a $d\times d$ symmetric matrix here)
and $p$ is a scalar pressure of fluid.
 $\mathrm{D}(u)$ is the symmetric part of the velocity gradient,
 \begin{equation*}
   \mathrm{D}(u)=\frac{1}{2}(\nabla u+(\nabla u)^T).
 \end{equation*}
  The $\mathrm{Q}$ above is a given bilinear form:
\begin{equation*}
 \mathrm{Q}(\tau,\nabla u)=\tau\Omega(u)-\Omega(u)\tau+b(\mathrm{D}(u)\tau+\tau \mathrm{D}(u)),
\end{equation*}
where $b$ is a parameter in $[-1,1],$ $\Omega(u)$ is the skew-symmetric part of $\nabla u,$ i.e.
\begin{equation*}
  \Omega(u)=\frac{1}{2}(\nabla u-(\nabla u)^T).
\end{equation*}
The parameters $\nu,\eta,\mu,\alpha$ are non-negative and they are specific to the characteristic of
the considered material, $\nu$ is the viscous coefficient, while $\eta$ is the stress coefficient. In \cite{LM}, $\mu$ and $\alpha$ correspond respectively to $1/We$ and $2(1-\theta)/(WeRe),$
where $Re$ is the Reynolds number, $\theta$ is the ratio between the relaxation and retardation times and $We$ is
the Weissenberg number. $k$ is the coupling coefficient connecting the velocity $u$ (kinetic energy) and the stress tensor $\tau$ (elastic potential energy).

\par
The Oldroyd-B model describes the motion of some viscoelastic flows. Formulations about viscoelastic flows of Oldroyd-B type are first
established by Oldroyd in \cite{OJ}. For more detailed physical background and derivations about this model, we refer the readers to \cite{BCF,PM,139,OJ}.
\par

When $\nu>0$ and $\eta=0$, Chemin and Masmoudi \cite{CM}
first obtained the local solutions and global small solutions  {in the critical Besov spaces} when $\nu>0, \mu_1>0, \alpha>0,$ and $\eta=0.$ { They get the global small solutions when the  initial and coupling parametra is small, i.e.($\mu_1\alpha\leq c\mu_2\nu$).
The condition $\mu_1\alpha\leq c\mu_2\nu$ means that coupling effect between the two equation is less important than the viscosity.}
{Inspired} by the work \cite{CW,CMZ}, Zi, Fang and Zhang improved their results in the critical $L^p$ framework for the case of non-small coupling parameters in \cite{ZFZ}. Zhu \cite{145} got small global smooth solutions of the $\mathrm{3D}$ Oldroyd-B model with $\eta=0, \mu=0$ by observing the linearization of the system satisfies the damped wave equation. Inspired by the work of Zhu \cite{145} and Danchin in \cite{DR}, Chen and Hao \cite{CH2} extended this small data global solution in Sobolev spaces to the critical Besov spaces. Moreover, Zhai \cite{zhai} constructs global solutions for a class of highly oscillating initial velocities by observing the special structure of the system. In the corotational case, i.e. $b=0,$ Lions and Masmoudi established the existence of global weak solution in \cite{LM}.

When $\nu=0$ and $\eta>0$, Elgindi and Rousset \cite{ER} established a global large solution in a certain sense by building a new quantity {to avoid singular operators.}
Later, Liu and Elgindi \cite{135} extend these results in $3d$ for totally small initial data $\|u_0,\tau_0\|_{H^s(\mathbb{R}^3)},~s>\frac{5}{2}$. Recently, Constantin, Wu, Zhao and Zhu \cite{wjh8,WZ2} established these small data global solutions in the case of no damping mechanism and  general {tensor} dissipation.

In this paper, we consider the global well-posedness, stability and instability for the Oldroyd-B model \eqref{gs} with $\nu=0$ and $\eta>0$. Without lose of generality, we let $\nu=0, a=1, \mu=1$ and $\eta=1$. Since the coupling coefficient $k$ is finite, we set $0\leq k\leq 10$ in this paper, then \eqref{gs} becomes:
\begin{equation}\label{1}
\begin{cases}
\partial_tu+(u\cdot\nabla)u+\nabla p=k\mathrm{div}(\tau),\\
\partial_t\tau+(u\cdot\nabla)\tau-\Delta\tau+\tau+\mathrm{Q}(\nabla u,\tau)=\mathbb{D}u,\\
\mathrm{div} u=0,\\
u(x,0)=u_0(x),~~~\tau(0,x)=\tau_0(x),
\end{cases}
\end{equation}
When $k>0$, since $\tau$ and $\mathbb{D}u$ are the damping terma, some dissipations will appear on $\|\tau\|_{H^s}$ and $\|\mathbb{D}u\|_{H^{s-1}}$. However, when $k=0$, since the system \eqref{1} decouples, all the dissipations will vanish. This implies that the coupling coefficient $k$ plays a key role in energy dissipation, which is what we study on this paper.

Firstly, we introduce the global existence of \eqref{1}. Recall that, for $d=2$, by building a new quantity $\Gamma=w-\frac{curl div}{\Delta}\tau$ Elgindi and Rousset \cite{ER} established a class of global solutions for \eqref{1}, which need the following initial conditions:
\begin{equation*}
  \|u_0,\tau_0\|_{H^1(\mathbb{R}^2)}+\|curlu_0,\tau_0\|_{B^0_{\infty,1}(\mathbb{R}^2)}\leq\epsilon_0
  ,~~(u_0,\tau_0)\in H^s(\mathbb{R}^2)\times H^s(\mathbb{R}^2),~s>2.
\end{equation*}
Since $H^s\hookrightarrow B^{s+1}_{2,1}\hookrightarrow B^0_{\infty,1}$ with $s>1$, their result means some large initial data for the global existence. However, when $d=3$, it seems to be a challenge for the same conditions of initial data. Because a new term $w\nabla u$ appears in the equation of $w(t,x)$ in dimensional three, so as the equation of $\Gamma(t,x)$. This cause the main difficulty to obtain the global existence for \eqref{1}.

To overcome this difficulty, we observe that the damping term $\mathbb{D}u$ and $\tau$ can bring more damping effect for $w(t,x)$ (so as $\Gamma(t,x)$) when the coupling term $kdiv(\tau)$ is small enough. This will help us prove the global existence for a more general class of initial data such that:
\begin{equation}\label{small1}
  \|u_0\|_{B^1_{\infty,1}(\mathbb{R}^3)}+\|\tau_0\|_{B^0_{\infty,1}(\mathbb{R}^3)}\leq k^4\epsilon_0,~~\forall  k\in(0,\frac{1}{C^2+1}].
\end{equation}
By \eqref{small1} we obtain the global existence of \eqref{1} without $(u_0,\tau_0)\in H^s$. Indeed,  \cite{ER} used the following estimation ($\tilde{R}:=-(-\Delta)^{-1}curl(div(\cdot)),~~\forall \epsilon>0$):
\begin{align}\label{neednot}
\|[\tilde{R},u\cdot\nabla]\tau\|_{L^2(B^0_{\infty,1})}&\leq
C\|w\|_{L_t^{\infty}(L^{\infty}\cap L^2)}\|\tau\|_{L^2_t(B^{\epsilon}_{\infty,1}\cap L^2)}\notag\\
&\leq C\|w\|_{L_t^{\infty}(L^{\infty}\cap L^2)}\|\tau\|_{L^2_t(H^2)} ,~~
\end{align}
where $H^2(\mathbb{R}^2)\hookrightarrow B^{\epsilon}_{\infty,1}(\mathbb{R}^2)$. With the help the convective term $u\nabla \Gamma$, we find that the ${H^s}$ norms for $w,\tau$ are not required in \eqref{neednot}. So our condition
 \eqref{small1} implies a more general class of large initial data for global existence (see Remark \ref{small11}). Moreover,
 for sufficient small $k$, we obtain the exponential decay in the critical Besov spaces. Here are two results of global existence.
  \begin{theorem}\label{Th2}
 Let $(u_0,\tau_0)\in B^{1+\frac{3}{p}}_{p,1} (\mathbb{R}^3)\times B^{\frac{3}{p}}_{p,1}(\mathbb{R}^3)$ with $p\in [1,\infty]$ . If there exists a $\epsilon_0$ small enough such that
 $$\|u_0\|_{B^{1}_{\infty,1}}+\|\tau_0\|_{B^{0}_{\infty,1}}\leq k^4\epsilon_0:=\frac{k^4}{4(C^6+1)} ,~~~\forall k\in (0,\frac{1}{C^2+1}],$$
 then the solution $(u,\tau)$ of \eqref{1} exists globally in $C([0,\infty);B^{1+\frac{d}{p}}_{p,1}(\mathbb{R}))\times C([0,\infty);B^{\frac{d}{p}}_{p,1}(\mathbb{R}))\cap L^1([0,\infty);B^{2+\frac{d}{p}}(\mathbb{R})))$. Moreover, one have
\begin{align}\label{edecay}
 \|\nabla u(t)\|_{B^{\frac{3}{p}}_{p,1}}+k\|\tau(t)\|_{B^{\frac{3}{p}}_{p,1}}
 \leq  C\|\nabla u_0,~\tau_0\|_{B^{\frac{3}{p}}_{p,1}}e^{-\frac{k}{4}t}.
 \end{align}
\end{theorem}
\begin{remark}\label{small11}
Since $B^{\frac{3}{p}}_{p,1}\hookrightarrow B^{0}_{\infty,1},~p<\infty$. By Theorem \ref{Th2} we claim that our result includes some large initial data. For example,
choose $\varphi$ be a smooth, radial and non-negative function in $\mathbb{R}^2$ such that
\begin{equation}
\phi=\begin{cases}
1,~~for~|\xi|\leq 1,\\
0,~~for~|\xi|\geq 2.
\end{cases}
\end{equation}
Let $(u_0,\tau_0):=\frac{1}{N}(\psi,\varphi)$, where $\psi,\varphi\in\mathbb{S}^3$, $div\psi=0$ and $F(\varphi)=(\phi(\xi-2^Ne),\phi(\xi-2^Ne),\phi(\xi-2^Ne))$ with $e=(1,1),N\in\mathbb{N}^+$. Then, one can easily deduce that
 $$\Delta_j\varphi=\varphi~~when~j=N;
 ~~~\Delta_j\varphi=0~~when~j\neq N.$$
So for sufficient large $N$ and $p<\infty$, we have
$$\|\tau_0\|_{B^{\frac{3}{p}}_{p,1}}\approx\frac{2^{\frac{3}{p}N}}{N},
~~but~~\|\tau_0\|_{B^{0}_{\infty,1}}\leq \frac{C}{N}.$$
This implies the global existence for some large initial data, which is different from the result in \cite{ER,wjh8}.
\end{remark}
 \begin{theorem}\label{th1}
 Let $(u_0,\tau_0)\in (H^{s} (\mathbb{R}^3),H^{s} (\mathbb{R}^3))$ with $s>\frac{5}{2}$. If there exists a $\epsilon_0$ small enough such that
 \begin{align}\label{smallness}
 \|\nabla u_0\|_{H^{s-1}}+\|\tau_0\|_{H^{s}}\leq k^6\epsilon_0= \frac{k^6}{4(C^6+1)} ,~~~\forall k\in (0,10]
 \end{align}
 then the solution $(u,\tau)$ of \eqref{1} exists globally in $C([0,\infty);H^{s}(\mathbb{R}))\times C([0,\infty);H^{s}(\mathbb{R}^3))\cap L^2([0,\infty);H^{s+1}(\mathbb{R}^3)))$. Moreover, one have
 \begin{align}\label{pdecay}
 \|\tau(t)\|_{H^s}+\|\nabla u(t)\|_{H^{s-1}}\leq C\epsilon_0(1+t)^{-\frac{1}{2}}.
 \end{align}
 \end{theorem}
 \begin{remark}\label{interesting}
\cite{135} proved the global existence when $\|u_0,\tau_0\|_{H^s}\leq \epsilon_0~(s>\frac{5}{2})$ , and the polynomial decay of $\|\tau(t)\|_{H^s}+\|\nabla u(t)\|_{H^{s-1}}.$ Theorem \ref{th1} just attenuates the condition such that $\|u_0\|_{L^2}$ could be large, a small improvement.

Combining Theorem \ref{th1} and Theorem \ref{Th2}, we find an interesting phenomenon. When the coupling coefficient $k$ is large ($k\in (0,10]$), by Theorem \ref{th1}, we obtain the polynomial decay of $\|\nabla u,\tau\|_{L^2}$. However, when $k$ is small ($k\in (0,\frac{1}{C^2+1}]$), by Theorem \ref{Th2}, we obtain the exponential decay of $\|\nabla u,\tau\|_{L^2}$. This implies that the size of the coupling coefficient $k$ determines the extent of the decay of the velocity field $u$ and the stress tensor $\tau$. There will be a better decay for sufficient small $k$, since small $k$ means small distraction for the equation of $u(t,x)$, while the damping term $\mathbb{D}u$ can develop a larger impact.
\end{remark}

Next, by Remark \ref{interesting},
$\bar{k}= \frac{1}{C^2+1}$ seems to be a boundary between these two kinds of attenuation. Furthermore, one will ask whether the solutions are close to each other when $k\rightarrow \bar{k}$? The answer is true. Now we give a more general theorem to verify that all the solutions in above theorem will be close to each other when $k^m\rightarrow k$ for any fixed $k>0$ and $t>0$. 
\begin{theorem}\label{th3}
 Let $(u_0,\tau_0)\in H^{s} (\mathbb{R}^3)\times H^{s} (\mathbb{R}^3)$ with $s>\frac{5}{2}$. Assume
  \begin{align}\label{th3.1}
\lim_{m\rightarrow \infty}|k^m-k|=0~~for~any~fixed~k,k^m\in(0,10].
 \end{align}
 If the initial data satisfies
 $$\|\nabla u_0\|_{H^{s-1}}+\|\tau_0\|_{H^{s}}\leq k^6\epsilon_0 ,$$
 then we have
 \begin{align}\label{th3.2}
 \lim_{m\rightarrow \infty}\|u^m-u\|_{L^{\infty}([0,\infty);H^{s})}+ \|\tau^m-\tau\|_{L^{\infty}([0,\infty);H^{s})\cap L^2([0,\infty);H^{s+1})}=0,
 \end{align}
 where $(u^m,\tau^m)$ are the global solutions of \eqref{1} with the coefficient $k^m$~($m\in\mathbb{N}\cap\infty$) and $(u^{\infty},\tau^{\infty}):=(u,\tau)$.
 \end{theorem}

However, the damping effect can not be better when the coupling coefficient $k$ is too small that $k\rightarrow 0$. Because \eqref{1}  will decouple as $k=0$, which means the damping effect will vanish! As a result, \eqref{th3.2} is no longer valid. Indeed, we will prove that the energy of $u^k$ will have a jump when $k\rightarrow 0$ for large time. This implies the system \eqref{1} is not globally steady for $k\rightarrow 0$, while for local time \eqref{1} is steady in \cite{luoshidi}. 

Set $\mathbb{A}:=\{(u_0,\tau_0)\in H^{s}(\mathbb{R}^3)\times H^{s} (\mathbb{R}^3),~s>\frac{5}{2}|~\eqref{1} has~a~unique~solution~for~any~fixed~k\}.$ Here is the unsteady result.

\begin{theorem}\label{th4}
 Let $(u^k,\tau^k)$ be the corresponding solutions for \eqref{1} with every $k\in [0,\epsilon_0]$.
 Then there exists a large time $T(k)$ and a sequence  $(u_0,\tau_0)(k)\in\mathbb{A}$ as initial data such that
when $t\geq T(k)$, we have
 $$ \|u-u^k\|_{L^2}\geq \frac{\epsilon_0}{2},$$
 where $\epsilon_0=\frac{1}{4(C^6+1)}$, a fixed constant.
 \end{theorem}
  \begin{remark}
When $k\in (0,\frac{1}{C^2+1}]$, by Theorem \ref{Th2}, we obtain the exponential decay of $\|u\|_{L^2}$. However, when $k=0$, by the classical Euler equation we deduce that $\|u(t)\|_{L^2}$ is conservative, while $\|\tau\|_{H^s}$ doesn't decay anymore. This implies that the sign of the coupling coefficient $k$ determines whether the norm of the velocity field $u(t,x)$ has decay. In fact, when $k>0$, since $\tau$ is a heat type equation with damping mechanisms $\tau$ and $\mathbb{D}u$, the coupling term $kdiv\tau$ passes the decay of $\tau$ to $u$, but this process of transformation is transient for $k\rightarrow 0$ in large time (see Theorem \ref{th4}).

All in all, combining Theorem 1.1-Theorem 1.4 we conclude that larger coupling coefficient $k$ will have a worse impact to the extent
of the decay of $(u,\tau)$, but it is necessary for the appearance of decay ($k$ must be positive, or the decay will vanish instantly), while the damping term $\tau$ and $\mathbb{D}u$ always bring the well impact for the decay.
\end{remark}

The paper is organized as follows. In section 2, we will give the tools(Littlewood-Paley decomposition and
paradifferential calculus) and Besov spaces. In section 3,  we prove the global existence of \eqref{1} for different kinds of $k$. In section 4, we prove the stability of \eqref{1} when $k^m\rightarrow k>0$. In section 5, we show that the energy of $u^k(t,x)$ will have a jump when $k\rightarrow 0$ for large time, which implies the \eqref{1} is not globally steady for $k\rightarrow 0$.

\par\noindent
{\bf Notation}
Throughout the paper, we denote the norms of usual Lebesgue space $L^p(\R^3)$ by $
\|u\|^p_{L^p}=\int_{\Omega}|u|^p dx,~\hbox{for}~1\leq p<\infty$.
$C_i$ and $C$ denote different
positive constants in different places.

\section{Preliminaries}
\par
In this section, we will recall some properties about the Littlewood-Paley decomposition and Besov spaces.
\begin{proposition}\label{p2}
Let $\mathcal{C}$ be the annulus $\{\xi\in\mathbb{R}^d:\frac 3 4\leq|\xi|\leq\frac 8 3\}$. There exist radial functions $\chi$ and $\varphi$, valued in the interval $[0,1]$, belonging respectively to $\mathcal{D}(B(0,\frac 4 3))$ and $\mathcal{D}(\mathcal{C})$, and such that
$$ \forall\xi\in\mathbb{R}^d,\ \chi(\xi)+\sum_{j\geq 0}\varphi(2^{-j}\xi)=1, $$
$$ \forall\xi\in\mathbb{R}^d\backslash\{0\},\ \sum_{j\in\mathbb{Z}}\varphi(2^{-j}\xi)=1, $$
$$ |j-j'|\geq 2\Rightarrow\mathrm{Supp}\ \varphi(2^{-j}\cdot)\cap \mathrm{Supp}\ \varphi(2^{-j'}\cdot)=\emptyset, $$
$$ j\geq 1\Rightarrow\mathrm{Supp}\ \chi(\cdot)\cap \mathrm{Supp}\ \varphi(2^{-j}\cdot)=\emptyset. $$
The set $\widetilde{\mathcal{C}}=B(0,\frac 2 3)+\mathcal{C}$ is an annulus, and we have
$$ |j-j'|\geq 5\Rightarrow 2^{j}\mathcal{C}\cap 2^{j'}\widetilde{\mathcal{C}}=\emptyset. $$
Further, we have
$$ \forall\xi\in\mathbb{R}^d,\ \frac 1 2\leq\chi^2(\xi)+\sum_{j\geq 0}\varphi^2(2^{-j}\xi)\leq 1, $$
$$ \forall\xi\in\mathbb{R}^d\backslash\{0\},\ \frac 1 2\leq\sum_{j\in\mathbb{Z}}\varphi^2(2^{-j}\xi)\leq 1. $$
\end{proposition}

\begin{definition}\cite{CH}
Let $u$ be a tempered distribution in $\mathcal{S}'(\mathbb{R}^d)$ and $\mathcal{F}$ be the Fourier transform and $\mathcal{F}^{-1}$ be its inverse. For all $j\in\mathbb{Z}$, define
$$
\Delta_j u=0\,\ \text{if}\,\ j\leq -2,\quad
\Delta_{-1} u=\mathcal{F}^{-1}(\chi\mathcal{F}u),\quad
\Delta_j u=\mathcal{F}^{-1}(\varphi(2^{-j}\cdot)\mathcal{F}u)\,\ \text{if}\,\ j\geq 0,\quad
S_j u=\sum_{j'<j}\Delta_{j'}u.
$$
Then the Littlewood-Paley decomposition is given as follows:
\begin{align*}
u=\sum_{j\in\mathbb{Z}}\Delta_j u \quad \text{in}\ \mathcal{S}'(\mathbb{R}^d).	
\end{align*}
Let $s\in\mathbb{R},\ 1\leq p,r\leq\infty.$ The nonhomogeneous Besov space $B^s_{p,r}(\mathbb{R}^d)$ is defined by
$$ B^s_{p,r}=B^s_{p,r}(\mathbb{R}^d)=\{u\in S'(\mathbb{R}^d):\|u\|_{B^s_{p,r}(\mathbb{R}^d)}=\Big\|(2^{js}\|\Delta_j u\|_{L^p(\mathbb{S}^d)})_j \Big\|_{l^r(\mathbb{Z})}<\infty\}. $$
\end{definition}
\begin{definition}\cite{CH}
The homogeneous dyadic blocks $\dot{\Delta}_j$ are defined on the tempered distributions by
\begin{equation*}
  \dot{\Delta}_ju=\varphi(2^{-j}D)u:=\mathcal{F}^{-1}(\varphi(2^{-j\cdot})\hat{u}).
\end{equation*}
\begin{equation*}
  \dot{S}_ju=\sum_{j'\leq j-1}\dot{\Delta}_{j'}u.
\end{equation*}
 \begin{definition}\label{de1}
 	  We denote by $S'_h$ the space of tempered distributions $u$ such that
 	 \begin{equation*}
 	 \lim_{j\rightarrow-\infty}\dot{S}_ju=0 ~~\text{in}~~S'.
 	 \end{equation*}

 \end{definition}
 The homogeneous Littlewood-Paley decomposition is defined as
 \begin{equation*}
 u=\sum_{j\in\mathbb{Z}}\dot{\Delta}_ju,~~~~\text{for}~~u\in S'_h.
 \end{equation*}
\end{definition}
\begin{definition}
  For $s\in\R, 1\leq p\leq\infty,$ the homogeneous Besov space $\dot{B}^s_{p,r}$ is defined as
\begin{equation*}
  \dot{B}^s_{p,r}:=\{u\in S'_h, \|u\|_{\dot{B}^s_{p,r}}<\infty\},
\end{equation*}
where the homogeneous Besov norm is given by
\begin{equation*}
  \|u\|_{\dot{B}^s_{p,r}}:=\|\{2^{js}\|\dot{\Delta}_ju\|_{L^p}\}_j\|_{l^r}.
\end{equation*}
\end{definition}

In this paper, we use the "time-space" Besov spaces or Chemin-Lerner space first introduced by Chemin and Lerner in \cite{CL}.
\par\noindent
{\bf Definition 2.3.}\,\,
{\it Let $s\in\R$ and $0<T\leq+\infty.$ We define
\begin{equation*}
  \|u\|_{\tilde{L}^q_T({B}^s_{p,1})}:=\sum_{j\in\mathbb{Z}}2^{js}\bigg(\int^T_0\|{\Delta}_ju(t)\|^q_{L^p}dt\bigg)^\frac{1}{q},
\end{equation*}
for $p,q\in[1,\infty)$ and with the standard modification for $p,q=\infty.$
}
\par\noindent
By the Minkowski's inequality, it is easy to verify that
\begin{equation*}
  \|u\|_{\tilde{L}^\lambda_T({B}^s_{p,r}))}\leq \|u\|_{L^\lambda_T({B}^s_{p,r}))}~~~\text{if}~~\lambda\leq r,
\end{equation*}
and
\begin{equation*}
  \|u\|_{\tilde{L}^\lambda_T({B}^s_{p,r}))}\geq \|u\|_{L^\lambda_T({B}^s_{p,r}))}~~~\text{if}~~\lambda\geq r.
\end{equation*}
\par\noindent
The following Bernstein's lemma will be repeatedly used in this paper.
 \begin{lemma}\label{le1}\cite{CH}
 	 Let $\mathcal{B}$ is a ball and $\mathcal{C}$ is a ring of $\R^d.$ There exists constant $C$ such that for any positive $\lambda,$
 	 any non-negative integer $k,$ any smooth homogeneous function $\sigma$ of degree $m,$
 	 any couple $(p,q)\in[1,\infty]^2$ with $q\geq p\geq1,$ and any function $u\in L^p,$ there holds
 	 \begin{equation*}
 	 \mathrm{supp}\hat{u}\subset\lambda\mathcal{B}\Rightarrow\sup_{|\alpha=k|}\|\partial^\alpha u\|_{L^q}\leq C^{k+1}\lambda^{k+d(\frac{1}{p}-\frac{1}{q})}\|u\|_{L^p},
 	 \end{equation*}
 	 \begin{equation*}
 	 \mathrm{supp}\hat{u}\subset\lambda\mathcal{C}\Rightarrow C^{-k-1}\lambda^{k}\|u\|_{L^p}\leq\sup_{|\alpha=k|}\|\partial^\alpha u\|_{L^p}\leq C^{k+1}\lambda^{k}\|u\|_{L^p},
 	 \end{equation*}
 	 \begin{equation*}
 	 \mathrm{supp}\hat{u}\subset\lambda\mathcal{C}\Rightarrow \sup_{|\alpha=k|}\|\sigma(D)u\|_{L^p}\leq C_{\sigma,m}\lambda^{m+d(\frac{1}{p}-\frac{1}{q})}\|u\|_{L^p}.
 	 \end{equation*}
\end{lemma}
Next, we will give the paraproducts and product estimates in Besov spaces. Recalling the paraproduct decomposition
\begin{equation*}
  uv={T}_uv+{T}_vu+{R}(u,v),
\end{equation*}
where
\begin{equation*}
  {T}_uv:=\sum_q{S}_{q-1}u{\Delta}_v,~~~{R}(u,v):=\sum_q{\Delta}_qu{\tilde{\Delta}}_qv,~~\text{and}~~{\tilde{\Delta}}_q
  ={\Delta}_{q-1}+{\Delta}_{q}+{\Delta}_{q+1}.
\end{equation*}
\par\noindent
The paraproduct ${T}$ and the remainder ${R}$ operators satisfy the following continuous properties.
\begin{proposition}\label{p2}\cite{CH}
  For all $s\in\R, \sigma>0,$ and $1\leq p, p_1,p_2,r,r_1,r_2\leq\infty,$ the paraproduct ${T}$ is a bilinear,
continuous operator from $L^\infty\times{B}^s_{p,r}$ to ${B}^s_{p,r}$ and from
${B}^{-\sigma}_{p_1,r_1}\times\dot{B}^{s}_{p_2,r_2}$ to ${B}^{s-\sigma}_{p,r}$ with $\frac{1}{r}=\min\{1,\frac{1}{r_1}+\frac{1}{r_2}\},
\frac{1}{p}=\frac{1}{p_1}+\frac{1}{p_2}.$
The remainder ${R}$ is bilinear continuous from ${B}^{s_1}_{p_1,r_1}\times {B}^{s_2}_{p_2,r_2}$
 to ${B}^{s_1+s_2}_{p,r}$ with $s_1+s_2>0, \frac{1}{p}=\frac{1}{p_1}+\frac{1}{p_2}\leq1,$
 and $\frac{1}{r}=\frac{1}{r_1}+\frac{1}{r_2}\leq 1.$ In particular, if $r=\infty,$
the continuous property for the remainder ${R}$ also holds for the case $s_1+s_2=0, r=\infty, \frac{1}{r_1}+\frac{1}{r_2}=1.$
\end{proposition}
Combining the above proposition with Lemma \ref{le1} yields the following product estimates:

\begin{coro}\label{co1}\cite{CH}
	 Let $a$ and $b$ be in $L^\infty\cap {B}^s_{p,r}$ for some $s>0$ and $(p,r)\in[1,\infty]^2.$
	Then there exists a constant $C$ depending only on $d, p$ and such that
	\begin{equation*}
	\|ab\|_{{B}^s_{p,r}}\leq C(\|a\|_{L^\infty}\|b\|_{{B}^s_{p,r}}+\|b\|_{L^\infty}\|a\|_{{B}^s_{p,r}}).
	\end{equation*}
\end{coro}
Finally, we intruduce some useful results about the following heat conductive equation and the transport equation
\begin{equation}\label{s1cuchong}
\left\{\begin{array}{l}
    u_t-\Delta u+\beta u=G,\ x\in\mathbb{R}^d,\ \beta\geq0,~t>0, \\
    u(0,x)=u_0(x),\ x\in\mathbb{R}^d,
\end{array}\right.
\end{equation}
\begin{equation}\label{s100}
\left\{\begin{array}{l}
    f_t+v\cdot\nabla f+\beta f=g,\ x\in\mathbb{R}^d,\ \beta\geq0,\ t>0, \\
    f(0,x)=f_0(x),\ x\in\mathbb{R}^d,
\end{array}\right.
\end{equation}
which are crucial to the proof of our main theorem later.
\begin{lemma}\label{priori estimate111}\cite{CH}
Let $1\leq p\leq q\leq\infty$ and $k\geq 0,$ it holds that
\begin{equation*}
  \|\nabla^ke^{t\Delta}f\|_{L^{q}}\leq Ct^{-\frac{k}{2}-\frac{1}{p}+\frac{1}{q}}\|f\|_{L^{p}}.
\end{equation*}
\end{lemma}

\begin{lemma}\label{heat}
Let $s\in\mathbb{R},~\beta\geq 0, 1\leq q,q_1,p,r\leq\infty$ with $q_1\leq q$. Assume $u_0$ in ${B}^s_{p,r}$, and $G$ in $\widetilde{L}^{q_1}_T(^s_{p,r})$. Then \eqref{s1cuchong} has a unique solution $u$ in $\widetilde{L}^{q}_T({B}^{s+\frac{2}{q}}_{p,r})$ and satisfies
\begin{equation}\label{heatg1}
\aligned
\|u\|_{\widetilde{L}^{q}_T({B}^{s+\frac{2}{q}}_{p,r})}\leq C_1\Big(\|u_0\|_{{B}^s_{p,r}}+(1+T^{1+\frac{1}{q}-\frac{1}{q_1}})\|G\|_{\widetilde{L}^{q_1}_T({B}^{s+\frac{2}{q_1}-2}_{p,r})}\Big).
\endaligned
\end{equation}
Moreover, if $\beta>0$, without loss of generality we set $\beta=1$, one have
\begin{equation}\label{heatg2}
\aligned
\|u\|_{\widetilde{L}^{q}_T({B}^{s+\frac{2}{q}}_{p,r})}\leq C_1\Big(\|u_0\|_{{B}^s_{p,r}}+\|G\|_{\widetilde{L}^{q_1}_T({B}^{s+\frac{2}{q_1}-2}_{p,r})}\Big), \endaligned
\end{equation}
and
\begin{equation}\label{heatg3}
\aligned
\|ue^{\theta t}\|_{\widetilde{L}^{q}_T({B}^{s+\frac{2}{q}}_{p,r})}\leq \frac{C_1}{1-\theta}\Big(\|u_0\|_{{B}^s_{p,r}}+\|e^{\theta t}G\|_{\widetilde{L}^{q_1}_T({B}^{s+\frac{2}{q_1}-2}_{p,r})}\Big), \endaligned
\end{equation}
where $0\leq\theta <1$.
\end{lemma}
\begin{proof}
\eqref{heatg1} can be founded in \cite{CH}, we should only prove \eqref{heatg2}. Indeed, since
$$\Delta_ju=e^{-t}e^{t\Delta}\Delta_ju_0
+\int_0^te^{-(t-s)}e^{(t-s)\Delta}\Delta_jGds,$$
when $j\geq 0$, by $\|e^{t\Delta}\Delta_ju\|_{L^{p}}\leq Ce^{-2^{2j}t}\|\Delta_ju\|_{L^{p}}$ one can easily get
$$\|2^{s+\frac{2}{q}}\|u\|_{L^{q}_TL^{p}}\|_{1_{j\geq0}l^r}\leq C_1\Big(\|u_0\|_{{B}^s_{p,r}}+\|G\|_{\widetilde{L}^{q_1}_T({B}^{s+\frac{2}{q_1}-2}_{p,r})}\Big).$$
When $j=-1$, by $\|e^{t\Delta}\Delta_{-1}u\|_{L^{p}}\leq C\|\Delta_{-1}u\|_{L^{p}}$ we have
$$\|\Delta_{-1}u\|_{L^{q}_TL^{p}}\leq C_1\Big(\|\Delta_{-1}u_0\|_{L^p}+\|\Delta_{-1}G\|_{{L}^{q_1}_T(L^p)}\Big).$$
Combining the above two inequality, we obtain \eqref{heatg2}. To prove \eqref{heatg3}, since
$$(e^{\theta t}\Delta_ju)=e^{-(1-\theta) t}e^{t\Delta}\Delta_ju_0
+\int_0^te^{-(1-\theta)(t-s)}e^{(t-s)\Delta}(e^{\theta s}\Delta_jG)ds,$$
one can take the similar operators to obtain \eqref{heatg3}.
\end{proof}

\begin{lemma}\label{priori estimate}\cite{CH}
Let $s\in [\max\{-\frac{d}{p},-\frac{d}{p'}\},\frac{d}{p}+1](s=1+\frac{d}{p},r=1; s=\max\{-\frac{d}{p},-\frac{d}{p'}\},r=\infty).$
There exists a constant $C$ such that for all solutions $f\in L^{\infty}([0,T];{B}^s_{p,r})$ of \eqref{s100} with initial data $f_0$ in ${B}^s_{p,r}$, and $g$ in $L^1([0,T];{B}^s_{p,r})$, we have, for a.e. $t\in[0,T]$,
\begin{equation}\label{g1}
\aligned
\|f(t)\|_{{B}^s_{p,r}}\leq& C\Big(\|f_0\|_{{B}^s_{p,r}}+\int_0^t V'(t')\|f(t')\|_{{B}^s_{p,r}}+\|g(t')\|_{{B}^s_{p,r}}dt' \Big)\\
\leq&
e^{C_2 V(t)}\Big(\|f_0\|_{{B}^s_{p,r}}+\int_0^t e^{-C_2 V(t')}\|g(t')\|_{{B}^s_{p,r}}dt'\Big),
\endaligned
\end{equation}
where $V(t)=\int_{0}^{t}\|\nabla v\|_{{B}^{\frac{d}{p}}_{p,r}\cap L^{\infty}}ds($if $s=1+\frac{1}{p},r=1$, $V'(t)=\int_{0}^{t}\|\nabla v\|_{{B}^{\frac{d}{p}}_{p,1}}ds).$
\end{lemma}

\begin{remark}\label{priori estimate1}\cite{CH}
If ${\rm{div}}v=0$, we can get the same result with a better indicator: $\max\{-\frac{d}{p},-\frac{d}{p'}\}-1<s<\frac{d}{p}+1($or $s=\max\{-\frac{d}{p},-\frac{d}{p'}\}-1,r=\infty).$
\end{remark}

\begin{lemma}\label{priori estimate0}
Let $\beta>0$.
There exists a constant $C$ such that for all smooth solutions of \eqref{s100} with initial data $f_0$ in ${L}^p$, and $\nabla v,g$ in $L^1([0,T];{L}^p)$, we have, for all $1\leq p\leq \infty$ and $t\in[0,T]$,
\begin{align}\label{damp1}
\|f(t)\|_{{L_t^1\cap L_t^P}({L}^p)}\leq& Ce^{V(t)}(\|f_0\|_{{L}^p}  +\int_0^t\|g(t')\|_{{L}^p}dt'),
\end{align}
where $V(t)=\int_{0}^{t}\|\nabla v(t)\|_{L^{\infty}}ds.$
\end{lemma}
\begin{proof}
With loss of generality, we set $\beta=1$. \eqref{s100} can be rewrite as
$$dt(e^tf)+u\nabla (e^tf)=(e^tg).$$
Then one can easily deduce that
\begin{align}
\|f(t)\|_{{L}^p}\leq& Ce^{V(t)}(\|f_0\|_{{L}^p} +\int_0^t\|e^{-(t-t')}g(t')\|_{{L}^p}dt'),
\end{align}
which implies \eqref{damp1} by Young inequality.
\end{proof}

\par\noindent
\section{Global existence}

\par\noindent

\textbf{The proof of Theorem \ref{Th2}:}
\begin{proof}
Generally speaking, the bootstrap argument starts with an assumption. Let $T^*$ be the maximal existence time of the solution, for any $0\leq t< T^*$,
\begin{equation}\label{cyz1}
 \|\nabla u\|_{L^1_T({B^{0}_{\infty,1}})}\leq k^2\epsilon_0,~~ \|\tau\|_{L^{\infty}_T(B^{0}_{\infty,1})\cap L^{1}_T(B^{2}_{\infty,1})}\leq k{\epsilon_0},~~
\epsilon_0:= \frac{1}{4^6(C^6+1)}.
\end{equation}
where $C$ is a fixed positive constant,~and $0\leq k\leq\frac{1}{4(C^2+1)}$.
Let the initial data $(u_0,\tau_0)$ be small enough such that
\begin{align}\label{globaljie1}
\|u_0\|_{B^1_{\infty,1}}+\|b_0\|_{B^0_{\infty,1}}\leq k^4\epsilon_0.
\end{align}
We will divide the proof into 4 sections.

(1). First, we give the estimation of $\|\nabla u\|_{L^\infty_T({B^{0}_{\infty,1}})}$.

Applying Lemma \ref{priori estimate} and \eqref{cyz1} to  the first equation of \eqref{1}, we have
\begin{equation}\label{l1}
  \|u\|_{L^\infty(B^{1}_{\infty,1})}\leq Ce^{Ck^2\delta}(\|u_0\|_{B^{1}_{\infty,1}}+k\|\tau\|_{L^1(B^{2}_{\infty,1})})
  \leq C(k^4\epsilon_0+k^2\epsilon_0)\leq Ck^2\epsilon_0,
\end{equation}


(2). Then, we estimate $\|\tau\|_{L^{\infty}_T(B^{0}_{\infty,1})\cap L^{1}_T(B^{2}_{\infty,1})}.$\\
Applying the Lemma \ref{heat} to the second equation of \eqref{gs2222}, it implies that
\begin{equation}\label{cyy20}
\aligned
 \|\tau\|_{L^{\infty}_T(B^{0}_{\infty,1})\cap L^{1}_T(B^{2}_{\infty,1})}&\leq\|\tau_0\|_{B^{0}_{\infty,1}}
 +\int^t_0\|Q(\nabla u,\tau)\|_{B^{0}_{\infty,1}}+\|u\cdot\nabla \tau\|_{B^{0}_{\infty,1}}+\|\nabla u\|_{B^{0}_{\infty,1}}ds\\
  &\leq C(\|\tau_0\|_{B^{0}_{\infty,1}}+\|u\|_{L^{\infty}_t( B^{1}_{\infty,1})}\|\tau\|_{L^1_t(B^{2}_{\infty,1})}
  +\|\nabla u\|_{L^1_t(B^{0}_{\infty,1})})\\
  &\leq C(k^4\epsilon_0+k^2\epsilon_0
  +k^2\epsilon_0\|\tau\|_{L^1_t(B^{2}_{\infty,1})}), \\
  &\leq C(k^4\epsilon_0+k^2\epsilon_0)\leq Ck^2\epsilon_0,
\endaligned
\end{equation}
where the last inequality holds by \eqref{cyz1} and \eqref{globaljie1}.

(3). Next, we estimate  $\|\nabla u\|_{L^1_T({B^{0}_{\infty,1}})}$ and complete the bootstrap argument.\\
We establish a new quantity\cite{}:
\begin{equation}\label{aa}
 \Gamma=w-k\tilde{R}\tau,~~~\tilde{R}=-(-\Delta)^{-1}curl(div(\cdot)),
\end{equation}
and get the following equation of $\Gamma:$
\begin{equation}\label{cyy21}
\left\{\begin{array}{l}
    \partial_t\Gamma+k\Gamma
 +u\cdot\nabla\Gamma-k[\tilde{R},u\cdot\nabla]\tau
 =w\nabla u-k\tilde{R}(Q(\nabla u,\tau)), \\
    \Gamma(0,x)=w_0-k\tilde{R}\tau_0.
\end{array}\right.
\end{equation}
Applying the $\Delta_j$ to the \eqref{cyy21},
note that
\begin{equation}\label{cyy25}
\aligned
  &\Delta_j(u\cdot\nabla \Gamma-k[\tilde{R},u\cdot\nabla]\tau)=\Delta_j(u\cdot\nabla w-k\tilde{R}(u\cdot\nabla\tau))\\
  &=\Delta_jT_u\nabla w-k\Delta_j\tilde{R}T_u\nabla\tau+f_j\\
  &=S_{j-1}u\nabla\Delta_jw+(\Delta_jT_u\nabla w-S_{j-1}u\nabla\Delta_jw)-(kT_u\nabla\triangle_j\tilde{R}\tau+k[\Delta_j\tilde{R},T_u\nabla]\tau)+f_j\\
  &=S_{j-1}u\nabla\Delta_j\Gamma+(\Delta_jT_u\nabla w-S_{j-1}u\nabla\Delta_jw)-k[\Delta_j\tilde{R},T_u\nabla]\tau+f_j\\
  &=S_{j-1}u\nabla\Delta_j\Gamma+(\Delta_jT_u\nabla w-S_{j-1}u\nabla\Delta_jw)-k[\Delta_j\tilde{R},T_u\nabla]\tau+f_j,
\endaligned
\end{equation}
where
\begin{equation*}
  f_j=\Delta_jT_{\nabla w}u+\Delta_jR(\nabla w,u)-k\Delta_j\tilde{R}T_{\nabla\tau}u-k\tilde{R}\Delta_jR(u,\nabla\tau).
\end{equation*}
So we have
\begin{equation}\label{0cyy21}
 \aligned
 \partial_t\Delta_j\Gamma+K\Delta_j\Gamma
 +S_{j-1}u\nabla\Delta_j\Gamma+(\Delta_jT_u\nabla w-S_{j-1}u\nabla\Delta_jw)-K[\Delta_j\tilde{R},T_u\nabla]\tau+f_j
 =G_j,
 \endaligned
\end{equation}
where $G_j:=\Delta_j(w\nabla u)-k\tilde{R}\Delta_j(Q(\nabla u,\tau))$.

Firstly, we estimate the nonlinear terms of \eqref{0cyy21}. \\
By Lemma 10.25 in \cite{CH}, we get the commutator estimations:
\begin{align}\label{k1}
  &\quad\sum_{j}\|(\Delta_jT_u\nabla w-S_{j-1}u\nabla\Delta_jw)\|_{L^\infty}+k\|[\Delta_j\tilde{R},T_u\nabla]\tau\|_{L^\infty}\notag\\
  &\leq C(k+1)\|\nabla u\|_{B^{0}_{\infty,1}}(\|w\|_{B^{0}_{\infty,1}}+\|\tau\|_{B^{2}_{\infty,1}}),
\end{align}
By Bony decomposition $f_j$ and $G_j$ can be estimated as
\begin{align}\label{k2}
  \sum_{j}\|f_j\|_{L^\infty}
  &\leq\sum_{j}\|\Delta_j T_{\nabla w}u\|_{L^\infty}
  +k\sum_{j}\|\Delta_j\tilde{R}T_{\nabla\tau}u\|_{L^\infty}
  +\|\Delta_j R(u,\nabla\tau)\|_{L^\infty}
  +\|\Delta_j R(u,\nabla\tau)\|_{L^\infty}\notag\\
  &\leq C(1+k)\|\nabla u\|_{B^{0}_{\infty,1}}(\|\tau\|_{B^{2}_{\infty,1}}
  +\|u\|_{{B}^{1}_{\infty,1}})
\end{align}
and
\begin{align}\label{k3}
  \sum_{j}\|G_j\|_{L^\infty}
 \leq C(1+k)\|\nabla u\|_{B^{0}_{\infty,1}}(\|\tau\|_{B^{2}_{\infty,1}}+\|u\|_{{B}^{1}_{\infty,1}}),
\end{align}
where we use the fact that $w\nabla u=div(w\otimes u)$ with $divw=divcurlu=0$ in three dimension.

Then, applying Lemma \ref{priori estimate0} with $p=\infty$ to \eqref{0cyy21} and taking $\sum_{j\geq -1}$, by \eqref{k1}-\eqref{k3} we deduce that
\begin{equation}\label{cyy22}
 \aligned
   \|\Gamma\|_{L^{\infty}_T(B^{0}_{\infty,1})}+k\|\Gamma\|_ {L^{1}_T(B^{0}_{\infty,1})}&\leq C(\|\Gamma_0\|_{B^{0}_{\infty,1}}
   +\int^t_0C(1+k)
   (\|\nabla u\|_{B^{0}_{\infty,1}}+\|\tau\|_{B^{2}_{\infty,1}})
   \|\nabla u\|_{B^{0}_{\infty,1}}ds\\
   &\leq C(k^4\epsilon_0+(1+k)(k^2\epsilon_0)^2)
   \leq Ck^4\epsilon_0
 \endaligned
\end{equation}
So we have
$$\|\Gamma\|_ {L^{1}_T(B^{0}_{\infty,1})}\leq Ck^3\epsilon_0.$$
Combining \eqref{aa}, we deduce that
\begin{align}\label{aaa}
\|u\|_{L^1_T({B^{0}_{\infty,1}})}
  &=C(k\|\tau\|_{L^1_T({B^{0}_{\infty,1}})}+\|\Gamma\|_{L^1_T({B^{0}_{\infty,1}})})\notag\\
  &\leq C(k^3\epsilon_0
  +k^3\epsilon_0)\notag\\
  &\leq \frac{1}{2}k^2\epsilon_0 ,
\end{align}
where $\epsilon_0$ and $k$ satisfies \eqref{cyz1} and \eqref{globaljie1}. Using the bootstrap argument for \eqref{aa} and \eqref{aaa}, we obtain that
\begin{equation}\label{cyz11}
 \|\nabla u\|_{L^1_T({B^{0}_{\infty,1}})}\leq k^2\delta~~and~~ \|\tau\|_{L^{\infty}_T(B^{0}_{\infty,1})\cap L^{1}_T(B^{2}_{\infty,1})}\leq k{\delta},~~\forall t\in [0, T^*).
\end{equation}

Now, one can obtain the global existence of $(u,\tau)$ in $ C([0,\infty); {B}^{1+\frac{2}{p}}_{p,1})\times \Big(C([0,\infty);{B}^{\frac{2}{p}}_{p,1})\cap L^1\big([0,\infty);{B}^{\frac{2}{p}+2}_{p,1}\big)\Big)$ easily, since \eqref{cyz11} can be the blow-up criteria for \eqref{1}. Indeed, applying Lemma \ref{heat}--\ref{priori estimate} to \eqref{1}, we have
\begin{align}\label{see1}
\|u\|_{L^{\infty}_{t}({B^{1+\frac{2}{p}}_{p,1}})}\leq & \|u_0\|_{{B^{1+\frac{2}{p}}_{p,1}}}
+C\int_{0}^{t}\|u\|_{B^1_{\infty,1}}\|u\|_{B^{1+\frac{2}{p}}_{p,1}}
+k\|\tau\|_{B^{1+\frac{2}{p}}_{p,1}}ds,
\end{align}
and
\begin{align}\label{see2}
\|\tau\|_{L^{\infty}_t(B^{\frac{2}{p}}_{p,1})\cap L^{1}_t(B^{2+\frac{2}{p}}_{p,1})}\leq & \|\tau_0\|_{B^{\frac{2}{p}}_{p,1}}
+C\int_0^t\|u\|_{L^2}\|\tau\|_{B^{2+\frac{2}{p}}_{p,1}}
+\|\tau\|_{L^\infty}\|u\|_{B^{1+\frac{2}{p}}_{p,1}}+\|u\|_{B^{1+\frac{2}{p}}_{p,1}}dt
\end{align}
Combining \eqref{see1}-\eqref{see1} with Gronwall inequality, we obtain
\begin{align}\label{see2}
\|u\|_{L^{\infty}_{t}({B^{1+\frac{2}{p}}_{p,1}})}+\|\tau\|_{L^{\infty}_t(B^{\frac{2}{p}}_{p,1})\cap L^{1}_t(B^{2+\frac{2}{p}}_{p,1})}\leq & C(\|u_0\|_{{B^{1+\frac{2}{p}}_{p,1}}}+\|\tau_0\|_{B^{\frac{2}{p}}_{p,1}})e^{Ct}
\leq Ce^t,~~\forall t\in [0, T^*).
\end{align}
This implies $T^*=\infty$.

(4) Finally, to complete the proof of Theorem \ref{Th1}, we now prove the exponential decay. Rewrite \eqref{cyy21} :
\begin{equation}\label{cyy21again}
 \aligned
 \partial_t(e^{kt}\Gamma)+u\cdot\nabla(e^{kt}\Gamma)-k[\tilde{R},u\cdot\nabla](e^{kt}\tau)
 =w\nabla (e^{kt}u)-k\tilde{R}(Q(\nabla u,(e^{kt}\tau))).
 \endaligned
\end{equation}
Since $w=\Gamma+k\tilde{R}\tau$, applying \eqref{priori estimate} in Lemma \ref{heat} one can obtain
\begin{equation}\label{cyy22again}
 \aligned
   \|e^{kt}\Gamma(t)\|_{B^{\frac{3}{p}}_{p,1}}&\leq C(\|\Gamma_0\|_{B^{\frac{3}{p}}_{p,1}}
   +\int^t_0
\|\nabla u\|_{B^{0}_{\infty,1}}(\|e^{ks}\nabla u\|_{B^{\frac{3}{p}}_{p,1}}+\|e^{ks}\tau\|_{B^{\frac{3}{p}+2}_{p,1}})
ds)\\
   &\leq C(k^4\epsilon_0+
k^2\epsilon_0\|e^{kt}\tau\|_{L^{1}_t(B^{\frac{3}{p}+2}_{p,1})}
   +\int^t_0\|\nabla u\|_{B^{0}_{\infty,1}}\|e^{ks}\Gamma\|_{B^{\frac{3}{p}}_{p,1}}ds)
 \endaligned
\end{equation}
where $0<k\leq\frac{1}{C^2+1}\leq\frac{1}{16}$ and
$\|\nabla u\|_{L^1_T(B^{0}_{\infty,1})}\leq k^4\epsilon_0$.

Recall the equation of $\tau$ in \eqref{1}, one have
\begin{equation}\label{tauagain}
 \aligned
 \partial_t(e^{t}\tau)-\Delta(e^{t}\tau)+u\cdot\nabla(e^{t}\tau)
 +Q(\nabla u,(e^{kt}\tau))=e^{t}\mathbb{D}u.
 \endaligned
\end{equation}
That is
$$e^{kt}\tau=e^{t\Delta}e^{-(1-k)t}\tau_0
+\int_{0}^{t}e^{(t-s)\Delta}e^{-(1-k)(t-s)}(e^{ks}F),$$
where $F:=-u\cdot\nabla(\tau)
 -Q(\nabla u,(\tau))+\mathbb{D}u$. Applying \eqref{heatg3} with $k=\theta$ in Lemma \ref{heat}, we have
 \begin{equation}\label{tau1again}
 \aligned
   \|e^{kt}\tau(t)\|_{B^{\frac{3}{p}}_{p,1}}+\|e^{ks}\tau\|_{L^{1}_t(B^{\frac{3}{p}+2}_{p,1})}&\leq C(\|\tau_0\|_{B^{\frac{3}{p}}_{p,1}}
   +\int^t_0C
\|u\|_{B^{1}_{\infty,1}}\|e^{kt}\tau\|_{B^{\frac{3}{p}+2}_{\infty,1}}+
   \|e^{kt}\nabla u\|_{B^{\frac{3}{p}}_{p,1}}ds\\
   &\leq C(\|\tau_0\|_{B^{\frac{3}{p}}_{p,1}}
   +Ck\|e^{ks}\tau\|_{L^1_T(B^{\frac{3}{p}+2}_{p,1})} +\int_0^t\|e^{kt}\Gamma\|_{B^{\frac{3}{p}}_{p,1}}ds)\\
   &\leq C(k^4\epsilon_0+\int_0^t\|e^{kt}\Gamma\|_{B^{\frac{3}{p}}_{p,1}}ds),
 \endaligned
\end{equation}
Combining \eqref{tauagain} with $\frac{k}{4(C+1)}\times$\eqref{tau1again}, and applying Gronwall inequality, we obtain
 \begin{equation}
 \aligned
   e^{kt}(k\|\tau(t)\|_{B^{\frac{3}{p}}_{p,1}}+\|\Gamma\|_{B^{\frac{3}{p}}_{p,1}})
\leq C(\|\tau_0+\nabla u_0\|_{B^{\frac{3}{p}}_{p,1}})e^{\frac{k}{4}t}\leq Ck^3e^{\frac{k}{4}t},
 \endaligned
\end{equation}
where we use $\epsilon_0\leq\frac{1}{4(C^6+1)}$ and $k\leq \frac{1}{C^2+1}$ by \eqref{cyz1}. Since $w=\Gamma+k\tilde{R}\tau$, we obtain
$$k\|\tau(t)\|_{B^{\frac{3}{p}}_{p,1}}+\|w(t)\|_{B^{\frac{3}{p}}_{p,1}}
\leq C(\|\tau_0+\nabla u_0\|_{B^{\frac{3}{p}}_{p,1}})e^{-\frac{k}{2}t}.$$
This complete the proof of Theorem \ref{Th2}.
\end{proof}

{\bf The proof of Theorem \ref{th1}:}\\
 The proof is similar to \cite{135}, we give the proof briefly.
 \begin{proof}
 For $0<k\leq 10$, assume that for any $0\leq t<T<T^*$ we have
$$\|\nabla u,\nabla\tau\|^2_{L^{\infty}_t(H^{s-1})}+\| \nabla\tau\|^2_{L^2_t(H^{s})}\leq k^4\delta^2
~~and~~\| \nabla u\|^2_{L^2_t(H^{s-1})}\leq k^4\delta^{\frac{3}{2}},$$
where $\delta:=16(C^2+1)\epsilon_0$ and $\epsilon_0:=\frac{1}{4(C^6+1)}$ for a fixed large constant $C$. Set the initial data such that
$$\|\nabla u\|^2_{H^{s-1}}+\|\tau_0\|^2_{H^{s}}\leq k^6\epsilon_0^2,$$

Firstly, taking the $L^2$ and $\dot{H}^1$ inner product of \eqref{1}, we have
\begin{align}\label{1.2}
\frac{1}{2}\|\tau\|^2_{L^{\infty}_t(L^2)}+\|\tau\|^2_{L^2_t(H^1)}&\leq \|\tau_0\|^2_{L^2}+\int_{0}^{t}\|\nabla u\|_{L^{\infty}}\|\tau\|^2_{L^2}+\frac{1}{2}\|\nabla u\|^2_{L^2}+\frac{1}{2}\|\tau\|^2_{L^2}ds\notag\\
&\leq Ck^6(\epsilon_0+\delta^{\frac{5}{2}}),
\end{align}
\begin{align}\label{1.1}
\|u\|^2_{L^{\infty}_t(L^2)}+k\|\tau\|^2_{L^{\infty}_t(L^2)}+k\|\tau\|^2_{L^2_t(H^1)}&\leq \|u_0\|^2_{L^2}+k\|\tau_0\|^2_{L^2}+k\int_{0}^{t}\|\nabla u\|_{L^{\infty}}\|\tau\|^2_{L^2}ds\notag\\
&\leq (\|u_0\|_{L^2}+k\|\tau_0\|_{L^2}+k\delta\|\tau\|^2_{L^2_t(L^2)})\notag\\
&\leq Ck^2(\|u_0\|_{L^2}+k\|\tau_0\|_{L^2})\leq Ck^4\epsilon_0,
\end{align}

\begin{align}\label{1.3}
\int_{0}^{t}\|\nabla u\|^2_{L^2}ds &\leq C\int_{0}^{t}(\|\tau\|^2_{H^2}+\|u\|^2_{W^{1,\infty}}\|\tau\|^2_{H^1}+ \|\nabla\tau\|_{L^2}\|\nabla u\|^2_{H^{s-1}}+k\|\nabla \tau\|^2_{L^2})ds+\|\tau\|_{L^2}\|\nabla u\|_{L^2} \notag\\
&\leq Ck^4\delta^2
\end{align}
and
\begin{align}\label{1.4}
&\quad \|\nabla u\|^2_{L^{\infty}_t(L^2)}+k\|\nabla\tau\|^2_{L^{\infty}_t(L^2)}
+k\|\nabla\tau\|^2_{L^2_t(H^1)}\notag\\
&\leq \|\nabla u_0\|^2_{L^2}+k\|\nabla\tau_0\|^2_{L^2}+\int_{0}^{t}a\|\nabla u\|_{L^{\infty}}\|\nabla u\|^2_{L^2}+k\|\nabla u\|_{L^{\infty}}\|\tau\|^2_{H^2}ds\notag\\
&\leq Ck^6(\epsilon_0+\delta^{\frac{5}{2}}).
\end{align}

Then, taking the $\dot{H}^s$ inner product of \eqref{1}, we have
\begin{align}\label{1.5}
&\quad \| u\|^2_{L^{\infty}_t(\dot{H}^s)}+k\|\tau\|^2_{L^{\infty}_t(\dot{H}^s)}+k\|\tau\|^2_{L^2_t(\dot{H}^s\cap\dot{H}^{s+1})}\notag\\
&\leq \| u_0\|^2_{\dot{H}^s}+k\| \tau_0\|^2_{\dot{H}^s}+\int_{0}^{t}a\|\nabla u\|_{L^{\infty}}\|\nabla u\|^2_{\dot{H}^s}+k\|\nabla u\|_{H^{s-1}}\|\tau\|^2_{H^{s+1}}ds\notag\\
&\leq Ck^6(\epsilon^2_0+\delta^{\frac{5}{2}}).
\end{align}

\begin{align}\label{1.6}
\int_{0}^{t}\|\nabla u\|^2_{\dot{H}^{s-1}}ds &\leq C\int_{0}^{t}(\|\tau\|^2_{H^{s}}+\|u\|^2_{W^{1,\infty}}\|\nabla\tau\|^2_{H^s})+\|\nabla u\|^2_{H^{s-1}}\|\nabla\tau\|^2_{H^s})\notag\\
&~~+k\|\nabla \tau\|^2_{L^2}ds+ \|\tau\|_{\dot{H}^{s}}\|\nabla u\|_{\dot{H}^{s-1}} \notag\\
&\leq Ck^4\delta^2
\end{align}
Combining \eqref{1.3}-\eqref{1.6} with the bootstrap argument, we finally obtain for any $0\leq t<T^*$
$$\|\nabla u,\nabla\tau\|^2_{L^{\infty}_t(H^{s-1})}+\| \nabla\tau\|^2_{L^2_t(H^{s})}\leq \frac{1}{2}k^4\delta^2
~~and~~\| \nabla u\|^2_{L^2_t(H^{s-1})}\leq \frac{1}{2}k^4\delta^{\frac{3}{2}},$$

Finally, since the proof of \eqref{pdecay} can refer to \cite{135}, this complete the proof.
\end{proof}

\section{Global stability for $0<k\leq 10$}

In this section, we will give the prove of Theorem \ref{th3}.  Firstly, we give the global stability for \eqref{1} in a weaker space $H^{s-1}(\mathbb{R}^3)\times H^{s-1}(\mathbb{R}^3)$.
\begin{lemma}\label{low}
Let $(u_1,\tau_1) ,(u_2,\tau_2)$ be two global strong solutions of \eqref{1} in Theorem \ref{th1} with fixed $0<k\leq 10$, then for any $t>0$ we have
\begin{align}\label{3.0}
&\quad\|u^1-u^2,\tau^1-\tau^2\|^2_{L^{\infty}_t(H^{s-1})}+\|\tau^1-\tau^2\|^2_{L^2_t(H^{s})}
+\|\nabla (u^1-u^2)\|^2_{L^2_t(H^{s-2})}\notag\\
&\leq \frac{C}{k}\|u^1_0-u^2_0,\tau^1_0-\tau^2_0\|^2_{H^{s-1}}.
\end{align}

\end{lemma}
\begin{proof}
Give the equation of $(u^1-u^2,\tau_1-\tau_2)$:
\begin{align}\label{3.19}
\left\{
\begin{array}{ll}
(u^1-u^2){t}+u^1\nabla (u^1-u^2)+(u^1-u^2)\nabla u^2+\nabla (P^1-P^2)=kdiv (\tau^1-\tau^2) ,\\[1ex]
(\tau^1-\tau^2)_{t}+(\tau^1-\tau^2)-\Delta (\tau^1-\tau^2)+u^1\nabla (\tau^1-\tau^1)+(u^1-u^2)\nabla\tau^1 \\[1ex]
+Q(\nabla (u^1-u^2),~\tau^1)+Q(\nabla u^1,~(\tau^1-\tau^2))=\mathbb{D}(u^1-u^2)\\[1ex]
\end{array}
\right.
\end{align}
By Theorem \ref{th1}, we have $$\|\nabla u_i,\nabla\tau_i\|^2_{L^{\infty}_t(H^{s-1})}+\| \nabla\tau_i\|^2_{L^2_t(H^{s})}\leq C\delta^2
~~and~~\| \nabla u_i\|^2_{L^2_t(H^{s-1})}\leq C\delta^{\frac{3}{2}},~~i=1,2,~~\forall t>0,$$
where $\delta:=\frac{\epsilon_0}{16(C^2+1)}$ and $\epsilon_0:=\frac{1}{4(C^6+1)}$.

Similarly to \eqref{1.1},\eqref{1.3},\eqref{1.5} and \eqref{1.6} in the proof of Theorem \ref{th1}, we also use the energy method and have:
\begin{align}\label{4.1}
&\quad \|u^1-u^2\|^2_{L^{\infty}_t(L^2)}+k\|\tau^1-\tau^2\|^2_{L^{\infty}_t(L^2)}+k\|\tau^1-\tau^2\|^2_{L^2_t(H^1)}\notag\\
&\leq \| u^1_0-u^2_0\|_{\dot{H}^s}+Ck\|\tau^1_0-\tau^2_0\|_{\dot{H}^s}\notag\\
&~~+\int_{0}^{t}\|u^2\|_{L^6}\|\nabla(u_1-u_2)\|_{L^2}\|u_1-u_2\|_{L^3}
+k(\ |\tau_2\|_{H^{s}}\|u_1-u_2\|_{L^{2}}\|\tau_1-\tau_2\|_{H^{1}}
+\|u_2\|_{H^{s}}\|\tau_1-\tau_2\|^2_{H^{1}})ds\notag\\
&\leq \|u^1_0-u^2_0\|^2_{L^2}+k\|\tau^1_0-\tau^2_0\|^2_{L^2}\notag\\
&~+\int_{0}^{t}
\|u^2\|^{\frac{2}{3}}_{L^6}\|u^1-u^2\|^2_{\dot{H}^{1}}
+\|u^2\|^2_{L^6}\|u^1-u^2\|^2_{L^2}ds
+k\delta(\|\tau^1-\tau^2\|^2_{L^{2}_t(H^1)}
+\|u^1-u^2\|_{L^{\infty}_t(L^2)})ds\notag\\
&\leq C(\|u^1_0-u^2_0\|^2_{L^2}+\|\tau^1_0-\tau^2_0\|^2_{L^2}
+\delta^{\frac{2}{3}}\|u^1-u^2\|^2_{L^{2}_t(\dot{H}^{1})}),
\end{align}
\begin{align}\label{4.3}
\int_{0}^{t}\|\nabla (u^1-u^2)\|^2_{L^2}ds &\leq C\int_{0}^{t}(\|\tau^1-\tau^2\|^2_{H^s}
+\|u_1-u_2\|^2_{H^{s-1}})(\|\tau_2,\tau_1\|_{H^s}+\|\nabla u_2,\nabla u_1\|_{H^{s-1}})\notag\\
&~~+\|\tau_1-\tau_2\|_{{H}^2}ds+C\|\tau_1-\tau_2\|_{L^2}\|u_1-u_2\|_{\dot{H}^1}\notag\\
&\leq C[\delta(\|\tau^1-\tau^2\|^2_{L^{2}_t(H^s)}
+\|u_1-u_2\|^2_{L^{2}_t(H^{s-1})})
+\|\tau_1-\tau_2\|^2_{L^{\infty}_t(L^2)}
+\|u_1-u_2\|^2_{L^{\infty}_t(\dot{H}^1)}].
\end{align}
\begin{align}\label{4.5}
&\quad \| u_1-u_2\|^2_{L^{\infty}_t(\dot{H}^{s-1})}
+k\|\tau^1-\tau^2\|^2_{L^{\infty}_t(\dot{H}^{s-1})}
+k\|\tau\|^2_{L^2_t(\dot{H}^{s-1}\cap\dot{H}^{s})}\notag\\
&\leq C[\| u^1_0-u^2_0\|_{\dot{H}^s}+\|\tau^1_0-\tau^2_0\|_{\dot{H}^s}
+\delta^{\frac{2}{3}}\|\nabla(u_1-u_2)\|^2_{L^2_t(H^{s-2})}].
\end{align}
and
\begin{align}\label{4.6}
\int_{0}^{t}\|\nabla (u^1-u^2)\|^2_{\dot{H}^{s-2}}ds \leq C(\|\tau_1-\tau_2\|^2_{L^{\infty}_t(H^{s-1})\cap L^2_t(H^{s})}
+\|u^1-u^2\|^2_{L^{\infty}_t(H^{s-1})})
\end{align}

Then, let
[\eqref{4.1}+\eqref{4.5}]+$\frac{k}{16(C+1)(k+1)}$[\eqref{4.3}+\eqref{4.6}], since $0<k\leq 10$, we deduce that
\begin{align}\label{3.7}
\quad&\| u^1-u^2\|^2_{L^{\infty}_t({H}^{s-1})}+
k\|\tau^1-\tau^2\|^2_{L^{\infty}_t({H}^{s-1})}
+k\|\tau^1-\tau^2\|^2_{L^2_t({H}^{s})}+\|\nabla (u^1-u^2)\|^2_{L^2_t(H^{s-2})}\notag\\
&\leq C(\|\tau^1_0-\tau^2_0\|^2_{H^{s-1}}
+\|u^1_0-u^2_0\|^2_{H^{s-1}}).
\end{align}
This implies \eqref{3.0}.
\end{proof}

Secondly, we give the global stability in the original space
$H^{s-1}(\mathbb{R}^3)\times H^{s-1}(\mathbb{R}^3)$.
\begin{theorem}\label{th5}
 Let $(u_0,\tau_0)\in H^{s} (\mathbb{R}^3)\times H^{s} (\mathbb{R}^3)$ with $s>\frac{5}{2}$. Assume that $(u_0,\tau_0)$ satisfies the conditions in Theorem \ref{th1} such that
 $$\|\nabla u_0\|_{H^{s-1}}+\|\tau_0\|_{H^{s}}\leq k^6\epsilon_0,~~for~fixed~k\in(0,10]. $$
 If there exists a sequence $(u^n_0,\tau^n_0)\in (H^{s} (\mathbb{R}^3),H^{s} (\mathbb{R}^3))$ such that
 $$\lim_{n\rightarrow \infty}\|u^n_0-u_0,\tau^n_0-\tau_0\|_{H^s}=0,$$
then for any $t>0$ we have
\begin{align}\label{wending}
 \lim_{n\rightarrow \infty}\|u^n-u\|_{L^{\infty}_t([0,\infty);H^{s})}+ \|u^n-u\|_{L^{\infty}_t([0,\infty);H^{s})}=0.
 \end{align}
\end{theorem}

\begin{proof}
Since the smallness of $(u_0,\tau_0)$ and $\lim_{n\rightarrow \infty}\|u^n_0-u_0,\tau^n_0-\tau_0\|_{H^s}=0$, let
$(u^n_j,\tau^n_j)$ be the solutions of \eqref{1} with the initial data $(S_ju^n_0,S_j\tau^n_j)~(n\in \mathbb{N}\cup\infty)$, then by Theorem \ref{th1}~(also use the bootstrap argument), $(u^n_j,\tau^n_j)$ are global solutions. Moreover, one can deduce that
\begin{align}\label{wen1}
\|\nabla u^n_j,k\nabla\tau^n_j\|^2_{L^{\infty}_t(H^{s-1})}+k\| \nabla\tau^n_j\|^2_{L^2_t(H^{s})}\leq C\delta^2~~and~~\| \nabla u^n_j\|^2_{L^2_t(H^{s-1})}\leq C\delta^{\frac{3}{2}},
\end{align}
Since \eqref{wen1} is the blow-up criterion of \eqref{1}, we easily obtain
\begin{align}\label{3.9}
\|\nabla u^n_j,k\nabla\tau^n_j\|^2_{L^{\infty}_t(H^{s})}+k\| \nabla\tau^n_j\|^2_{L^2_t(H^{s+1})}+\| \nabla u^n_j\|^2_{L^2_t(H^{s})}&\leq C\| S_ju_0,S_j\tau_0\|^2_{L^{\infty}_t(H^{s+1})}\notag\\
&\leq C(2^j\|u_0,\tau_0\|_{L^{\infty}_t(H^{s})})^2.
\end{align}

Now, by Lemma \ref{low}, for fixed $0<k\leq 10$ we already have
\begin{align}\label{bu3.8}
\|u^n-u, \tau^n-\tau\|_{L^{\infty}_t([0,\infty);H^{s-1})}\leq
 \frac{C}{k}\|u^n_0-u_0, \tau^n_0-\tau_0\|_{L^{\infty}_t([0,\infty);H^{s-1})}\rightarrow 0.
\end{align}
 In order to verify \eqref{wending}, we should only prove the high frequency estimation $\| u^n-u,\tau^n-\tau\|_{L^{\infty}_t(H^s)}\rightarrow 0$.
 Our main idea is to estimate:
\begin{align}\label{3.8}
\| u^n-u,\tau^n-\tau\|^2_{L^{\infty}_t(\dot{H}^{s})}
&\leq \| u^n-u^n_j,\tau^n-\tau^n_j\|^2_{L^{\infty}_t(\dot{H}^{s})}
+\| u^n_j-u^{\infty}_j,\tau^n_j-\tau^{\infty}_j\|^2_{L^{\infty}_t(\dot{H}^{s})}\notag\\
&~~+\| u^{\infty}_j-u^{\infty},\tau^{\infty}_j-\tau^{\infty}\|^2_{L^{\infty}_t(\dot{H}^{s})},
\end{align}
where $u^{\infty}:=u,~\tau^{\infty}:=\tau$. The proof will be divided into three parts.

\textbf{(1) estimate $\| u^n_j-u^{\infty}_j,\tau^n_j-\tau^{\infty}_j\|^2_{L^{\infty}_t(\dot{H}^{s})}$ for fixed j}\\
Firstly, we give the equation of $(u^n_j-u^{\infty}_j,\tau^n_j-\tau^{\infty}_j)$:
\begin{align}\label{3.19}
\left\{
\begin{array}{ll}
(u^n_j-u^{\infty}_j){t}+u^n_j\nabla (u^n_j-u^{\infty}_j)+(u^n_j-u^{\infty}_j)\nabla u^{\infty}_j+\nabla (P^n_j-P^{\infty}_j)=kdiv (\tau^n_j-\tau^{\infty}_j) ,\\[1ex]
(\tau^n_j-\tau^{\infty}_j)_{t}+(\tau^n_j-\tau^{\infty}_j)-\Delta (\tau^n_j-\tau^{\infty}_j)+u^n_j\nabla (\tau^n_j-\tau^{\infty}_j)+(u^n_j-u^{\infty}_j)\nabla\tau^{\infty}_j \\[1ex]
+Q(\nabla (u^n_j-u_j),~\tau^{\infty}_j)+Q(\nabla u^n_j,~(\tau^n_j-\tau^{\infty}_j))=\mathbb{D}(u^n_j-u^{\infty}_j)\\[1ex]
\end{array}
\right.
\end{align}
By Lemma \ref{low} we easily get
\begin{align}\label{bu1}
&\quad\|u^n_j-u^{\infty}_j,\tau^n-\tau^{\infty}_j\|^2_{L^{\infty}_t(H^{s-1})}
+\|\tau^n_j-\tau^{\infty}_j\|^2_{L^2_t(H^s)}
+\|\nabla (u^n_j-u^{\infty}_j)\|^2_{L^2_t(H^{s-2})}\notag\\
&\leq \frac{C}{k}\|S_ju^n_0-S_ju^{\infty}_0,S_j\tau^n_0-S_j\tau^{\infty}_0\|^2_{H^{s-1}}\notag\\
&\leq \frac{C}{k}\|u^n_0-u^{\infty}_0,\tau^n_0-\tau^{\infty}_0\|^2_{H^{s-1}}
\rightarrow 0,~~n\rightarrow\infty.
\end{align}
Then, taking the $\dot{H}^{s}$ inner product of \eqref{1}(similar to \eqref{4.5} and \eqref{4.6}), by \eqref{3.9} we have
\begin{align}\label{3.10}
&\quad \| u^n_j-u^{\infty}_j\|^2_{L^{\infty}_t(\dot{H}^{s})}
+k\|\tau^n_j-\tau^{\infty}_j\|^2_{L^{\infty}_t(\dot{H}^{s})}
+k\|\tau^n_j-\tau^{\infty}_j\|^2_{L^2_t(\dot{H}^{s}\cap\dot{H}^{s+1})}\notag\\
&\leq C[\| u^n_0-u^{\infty}_0\|_{\dot{H}^s}+\|\tau^n_0-\tau^{\infty}_0\|_{\dot{H}^s}
+\delta^{\frac{2}{3}}(\|\nabla(u^n_j-u^{\infty}_j)\|^2_{L^2_t(H^{s-1})}
+2^j\|u^n_j-u^{\infty}_j\|^2_{L^{2}_t(L^{\infty})})].
\end{align}
and
\begin{align}\label{3.11}
\int_{0}^{t}\|\nabla (u^n_j-u^{\infty}_j)\|^2_{\dot{H}^{s-1}}ds
\leq C(\|\tau^n_j-\tau^{\infty}_j\|^2_{L^{\infty}_t(H^{s})\cap L^2_t(H^{s+1})}
+\|u^n_j-u^{\infty}_j\|^2_{L^{\infty}_t(H^{s})})
\end{align}
Combing \eqref{3.10} with $\frac{k}{16(C+1)(k+1)}$\eqref{3.11}, we finally obtain
\begin{align}\label{3.12}
\| u^n_j-u^{\infty}_j\|^2_{L^{\infty}_t({H}^{s})}+k\| \tau^n_j-\tau^{\infty}_j\|^2_{L^{\infty}_t({H}^{s})}
&\leq C(2^j+1)(\| u^n_0-u^{\infty}_0\|_{{H}^s}+\|\tau^n_0-\tau^{\infty}_0\|_{{H}^s}
+\| u^n_j-u^{\infty}_j\|^2_{L^{\infty}_t(\dot{H}^{s})}\notag\\
&~~
+k\|\tau^n_j-\tau^{\infty}_j\|^2_{L^{\infty}_t(\dot{H}^{s})}
+k\|\tau^n_j-\tau^{\infty}_j\|^2_{L^2_t(\dot{H}^{s}\cap\dot{H}^{s+1})})\notag\\
&\rightarrow 0,~~for~ fixed~ j .
\end{align}
\textbf{(2) estimate $\| u^n-u^n_j,\tau^n-\tau^n_j\|^2_{L^{\infty}_t(\dot{H}^{s})}
$ for any $n\in\mathbb{N}\cup\infty$}\\
We give the equation of $(u^n_j-u^n,\tau^n_j-\tau^n)$
\begin{align}\label{3.19}
\left\{
\begin{array}{ll}
(u^n_j-u^n){t}+u^n_j\nabla (u^n_j-u^n)+(u^n_j-u^n)\nabla u_n+\nabla (P^n_j-P_n)=kdiv (\tau^n_j-\tau^n),\\[1ex]
(\tau^n_j-\tau^n)_{t}+(\tau^n_j-\tau^n)-\Delta (\tau^n_j-\tau^n)+u^n_j\nabla (\tau^n_j-\tau^n)+(u^n_j-u^n)\nabla\tau^n\\[1ex]
+Q(\nabla (u^n_j-u^n),~\tau^n)+Q(\nabla u^n_j,~(\tau^n_j-\tau^n))=\mathbb{D}(u^n_j-u^n)\\[1ex]
\end{array}
\right.
\end{align}
The operators are similar to Lemma \ref{low}. The only difference is the high order term:
\begin{align}\label{3.13}
&\quad\int_0^t<\nabla^{s}[(u^n-u^n_j)\nabla u^n_j],\nabla^{s}(u^n-u^n_j)>ds\notag\\
&\leq C\int_0^t\|\nabla u^n_j\|_{L^{\infty}}\|u^n-u^n_j\|^2_{\dot{H}^{s}}
+C\|\nabla u^n_j\|_{{\dot{H}^{s}}}\|u^n-u^n_j\|_{L^{\infty}}\|u^n-u^n_j\|_{\dot{H}^{s}}ds\notag\\
&\leq C\int_0^t\delta\|u^n-u^n_j\|^2_{\dot{H}^{s}}
+\frac{1}{\delta}\|\nabla u^n_j\|^2_{{\dot{H}^{s}}}\|u^n-u^n_j\|^2_{L^{\infty}}ds\notag\\
&\leq C\int_0^t\delta\|u^n-u^n_j\|^2_{\dot{H}^{s}}
+\frac{1}{\delta}\|\nabla u^n_j\|^2_{{\dot{H}^{s}}}\|\nabla (u^n-u^n_j)\|^2_{{H}^{s-2}}ds\notag\\
&\leq C\delta\|u^n-u^n_j\|^2_{L^{\infty}(\dot{H}^{s})}
+\frac{1}{\delta}(2^{j}\delta)^2\|u^n_0-S_ju^n_0\|^2_{H^{s-1}}\notag\\
&\leq C\delta\|u^n-u^n_j\|_{\dot{H}^{s}}
+\delta\|u^n_0-S_ju^n_0\|^2_{H^{s}}.
\end{align}
where the fourth inequality holds by Lemma \ref{low} and \eqref{3.9}, and we use the fact that $\|u^n_0-S_ju^n_0\|^2_{H^{s-1}}\leq C2^{-j}\|u^n_0-S_ju^n_0\|^2_{H^{s}}.$

Then, similar to \eqref{4.5} and \eqref{4.6} in Lemma \ref{low}, we have
\begin{align}\label{3.14}
&\quad \| u^n-u^{n}_j\|^2_{L^{\infty}_t(\dot{H}^{s})}
+k\|\tau^n-\tau^{n}_j\|^2_{L^{\infty}_t(\dot{H}^{s})}
+k\|\tau^n-\tau^{n}_j\|^2_{L^2_t(\dot{H}^{s}\cap\dot{H}^{s+1})}\notag\\
&\leq C(\| u^n_0-S_ju^{n}_0\|_{\dot{H}^s}+\|\tau^n_0-S_j\tau^{n}_0\|_{\dot{H}^s}
+\delta^{\frac{3}{4}}(\|\nabla(u^n-u^{n}_j)\|^2_{L^2_t(H^{s-1})}
+\|u^n_0-S_ju^n_0\|^2_{H^{s}})).
\end{align}
and
\begin{align}\label{3.15}
\int_{0}^{t}\|\nabla (u^n-u^{n}_j)\|^2_{\dot{H}^{s-1}}ds
&\leq C(\|\tau^n-\tau^{n}_j\|^2_{L^{\infty}_t(H^{s})\cap L^2_t(H^{s+1})}
+\|u^n-u^{n}_j\|^2_{L^{\infty}_t(H^{s})})
\end{align}
Combing\eqref{3.14} with $\frac{k}{16(C+1)(k+1)}\times$\eqref{3.15}, we obtain
\begin{align}\label{3.16}
&\quad \| u^n-u^{n}_j\|^2_{L^{\infty}_t(\dot{H}^{s})}+k\|\tau^n-\tau^{n}_j\|^2_{L^{\infty}_t(\dot{H}^{s})}
+k\|\tau^n-\tau^{n}_j\|^2_{L^2_t(\dot{H}^{s}\cap\dot{H}^{s+1})}
+\int_{0}^{t}\|\nabla (u^n-u^{n}_j)\|^2_{\dot{H}^{s-1}}ds\notag\\
&\leq C(\| u^n_0-S_ju^{n}_0,\tau^n_0-S_j\tau^{n}_0\|_{{H}^s}+\delta\|\nabla(u^n-u^{n}_j)\|^2_{L^2_t(H^{s-1})}
)\notag\\
&\rightarrow 0,~~j\rightarrow\infty,~~\forall n\in\mathbb{N}^+\cup{\infty}.
\end{align}

\textbf{(3) Complete the proof}\\
Combing \eqref{3.16},\eqref{3.12} with \eqref{3.8}, one obtain that
$$\| u^n-u^{\infty},\tau^n-\tau^{\infty}\|_{L^{\infty}_t(\dot{H}^{s})}
\rightarrow 0,~~n\rightarrow 0.$$
In fact, for any $\epsilon>0$, by \eqref{3.16}, there exists a $M(\epsilon)$ such that, when $j\geq M$, we have
$$\| u^n-u^{n}_j,\tau^n-\tau^{n}_j\|_{L^{\infty}_t({H}^{s})}\leq \frac{\epsilon}{3},~~\forall n\in\mathbb{N}^+\cup\infty.$$
Then, for this $j$, by \eqref{3.12}, there exists a $\bar{M}(j,\epsilon)$ such that, when $n\geq \bar{M}$, we have
$$\| u^n_j-u^{\infty}_j,\tau^n_j-\tau^{\infty}_j\|_{L^{\infty}_t({H}^{s})}\leq \frac{\epsilon}{3}.$$
where $\bar{M}$ is dependent on $j,\epsilon$, since $j$ is dependent on $M(\epsilon)$, this implies that $\bar{M}$ is dependent on $\epsilon$.
Finally, we have
\begin{align}\label{3.17}
\| u^n-u,\tau^n-\tau\|^2_{L^{\infty}_t(\dot{H}^{s})}
&\leq \| u^n-u^n_j,\tau^n-\tau^n_j\|^2_{L^{\infty}_t(\dot{H}^{s})}
+\| u^n_j-u^{\infty}_j,\tau^n_j-\tau^{\infty}_j\|^2_{L^{\infty}_t(\dot{H}^{s})}\notag\\
&~~+\| u^{\infty}_j-u^{\infty},\tau^{\infty}_j-\tau^{\infty}\|^2_{L^{\infty}_t(\dot{H}^{s})}\notag\\
&\leq \frac{\epsilon}{3}+\frac{\epsilon}{3}+\frac{\epsilon}{3}=\epsilon.
\end{align}
Combining with \eqref{bu3.8}, that is
$$\| u^n-u^{\infty},\tau^n-\tau^{\infty}\|_{L^{\infty}_t(\dot{H}^{s})}
\rightarrow 0,~~n\rightarrow 0,$$
which completes the proof.
\end{proof}

Thanks to the globally steady result in Theorem \ref{th5}, now we can prove Theorem \ref{th3} easily.

\textbf{Proof of Theorem \ref{th3}:}
\begin{proof}
To prove
$$\lim_{m\rightarrow \infty}\|u^m-u,~\tau^m-\tau\|_{L^{\infty}([0,\infty);H^{s})}=0$$
for $k^m\rightarrow k,~m\rightarrow \infty.$
Our main idea is to estimate
\begin{align}\label{3.18}
\| u^m-u,\tau^m-\tau\|^2_{L^{\infty}_t({H}^{s})}
\leq \| u^m-u^m_j,\tau^m-\tau^m_j\|^2_{L^{\infty}_t({H}^{s})}
+\| u^m_j-u_j,\tau^m_j-\tau_j\|^2_{L^{\infty}_t({H}^{s})}
+\| u_j-u,\tau_j-\tau\|^2_{L^{\infty}_t({H}^{s})},
\end{align}
where $(u^m,\tau^m)$ are the solutions of \eqref{1} with the coefficient $k^m$ and the same initial data $(u_0,\tau_0)$;~ $(u^m_j,\tau^m_j)$ are the solutions of \eqref{1} with the coefficient $k^m$ and the same initial data $(S_ju_0,S_j\tau)~(m\in \mathbb{N}\cup\infty,~k^{\infty}=k,~u^{\infty}_{j}:=u_{j},~\tau^{\infty}_{j}:=\tau_{j})$.

Firstly, we estimate the term $\| u^m_j-u^{\infty}_j,\tau^m_j-\tau^{\infty}_j\|^2_{L^{\infty}_t(\dot{H}^{s})}
$ with fix $j$.\\ We have
\begin{align}\label{3.19}
\left\{
\begin{array}{ll}
(u^m_j-u_j)_{t}+u^m_j\nabla (u^m_j-u_j)+(u^m_j-u_j)\nabla u_j+\nabla (P^m_j-P_j)=kdiv (\tau^m_j-\tau_j)+(k^m-k)div\tau^m_j ,\\[1ex]
(\tau^m_j-\tau_j)_{t}+(\tau^m_j-\tau_j)-\Delta (\tau^m_j-\tau_j)+u^m_j\nabla (\tau^m_j-\tau_j)+(u^m_j-u_j)\nabla\tau_j\\[1ex]
+Q(\nabla (u^m_j-u_j),~\tau_j)+Q(\nabla u^m_j,~(\tau^m_j-\tau_j))=\mathbb{D}(u^m_j-u_j)\\[1ex]
\end{array}
\right.
\end{align}

Similar to the proof of Theorem \ref{th5}, by the energy estimations we have
\begin{align}\label{3.20}
\| u^m_j-u^{\infty}_j,\tau^m_j-\tau^{\infty}_j\|^2_{L^{\infty}_t({H}^{s})}
&\leq C(2^j+1)(\| u^m_0-u^{\infty}_0\|_{{H}^s}+\|\tau^m_0-\tau^{\infty}_0\|_{{H}^s}
+(k^m-k)\|div\tau^n_j\|^2_{L^{2}_t{H}^s})\notag\\
&\leq C(2^j+1)(\| u^m_0-u^{\infty}_0\|_{{H}^s}+\|\tau^m_0-\tau^{\infty}_0\|_{{H}^s}
+(k^m-k))\notag\\
&\rightarrow 0,~~m\rightarrow \infty,~~for~ fixed~ j .
\end{align}

Then, by Theorem \ref{th5}, we see that system \eqref{1} is globally steady for small initial data. Since $\| u_0-S_ju_0,\tau_0-S_j\tau_0\|_{\dot{H}^{s}}\rightarrow 0,~j\rightarrow \infty$, so we have
\begin{align}\label{3.21}
\| u^m-u^m_j,\tau^m-\tau^m_j\|^2_{L^{\infty}_t(\dot{H}^{s})}
\rightarrow 0,~~j\rightarrow \infty,~~for~any~m\in\mathbb{N}\cap\infty.
\end{align}

Finally, combining \eqref{3.20}, \eqref{3.21} with \eqref{3.18}, we deduce that
$$\lim_{n\rightarrow \infty}\|u^m-u\|_{L^{\infty}([0,\infty);H^{s})}+ \|u^m-u\|_{L^{\infty}([0,\infty);H^{s})}=0.$$
\end{proof}
\section{Instability for $k\rightarrow 0$.}

Indeed, by \eqref{3.0} in Lemma \ref{low}, one can see that $\|u^1-u^2,\tau^1-\tau^2\|^2_{L^{\infty}_t(H^{s-1})}$ can not be controlled by their initial data as $k\rightarrow 0$. In this section, we will prove that the system \eqref{1} is really unsteady as $k\rightarrow 0$ by showing that the $L^2$ norm of $u^k(t,x)$ will have a jump for large time.
\textbf{Proof of Theorem \ref{th4}:}\quad Let  $\epsilon_0=\frac{1}{64(C^4+1)} $ be the fixed small constant in Theorem \ref{Th2} and Theorem \ref{th1}. Recall the system:
\begin{equation}\label{th4-1}
\begin{cases}
\partial_tu+(u\cdot\nabla)u+\nabla p=k\mathrm{div}(\tau),\\
\partial_t\tau+(u\cdot\nabla)\tau-\Delta\tau+\tau+\mathrm{Q}(\nabla u,\tau)=\mathbb{D}u,~~k\in(0,\epsilon_0],\\
\mathrm{div} u=0,\\
u(x,0)=u_0(x),~~~\tau(0,x)=\tau_0(x),
\end{cases}
\end{equation}
and
\begin{equation}\label{th4-2}
\begin{cases}
\partial_tu+(u\cdot\nabla)u+\nabla p=0,\\
\partial_t\tau+(u\cdot\nabla)\tau-\Delta\tau+\tau+\mathrm{Q}(\nabla u,\tau)=\mathbb{D}u,\\
\mathrm{div} u=0,\\
u(x,0)=u_0(x),~~~\tau(0,x)=\tau_0(x),
\end{cases}
\end{equation}
To prove the instability when $k\rightarrow 0$, we first give the definition of the global stability:
$$\lim_{k\rightarrow 0}\|u^0-u^k\|_{L^{\infty}_t(B^{\frac{5}{2}}_{2,1})}
 +\|\tau^0-\tau^k\|_{L^{\infty}_T(B^{\frac{3}{2}}_{2,1})}=0,~~\forall (u_0,\tau_0)\in \mathbb{A} ~and~\forall t\in [0,\infty),$$
 where $\mathbb{A}:=\{(u_0,\tau_0)\in (B^{\frac{5}{2}}_{2,1} (\mathbb{R}^3),B^{\frac{3}{2}}_{2,1} (\mathbb{R}^3))|~\eqref{1}~ has~a~unique~solution~for~any~fixed~k\}.$
In order to prove the instability in large time, we should prove that
for any $k>0$~small enough, there exists a common initial sequence~ $(u_0,\tau_0)(k)$ and a $T(k)$ such that, when $t\geq T$, we have
\begin{align}\label{non1}
\|(u^0-u^a)(t)\|_{B^{\frac{5}{2}}_{2,1}}\geq \frac{\epsilon_0}{2},
\end{align}
Now, let an axisymmetric vector field  $\phi\in \mathbb{S}^3$ with $div\phi=0$. Set the initial data $$(u_0,\tau_0)(a)=k^{6}\epsilon_0 (\frac{\phi(k^4x)}{\|\phi\|_{L^2}}, 0).$$
For any $0< k\leq  \epsilon_0$, we have
$$\|u_0\|_{L^2}=\epsilon_0,~
~~and~~\|u_0\|_{B^1_{\infty,1}}
\leq \|u_0\|_{L^{\infty}}+\|u_0\|_{\dot{H}^3}\leq Ck^{6}\epsilon_0,~~w_0=curlu_0.$$
These satisfy the conditions in Theorem \ref{th1} and Theorem \ref{Th2}, which means $(u_0,\tau_0)\in\mathbb{A}$.

On one hand, by Theorem \ref{Th2}, \eqref{th4-1} has a unique global strong solution $(u^k,\tau^k)$ with the initial data $(u_0,\tau_0)(k)$~($\tau_0=0\nRightarrow\tau=0$). We also obtain the $L^2$ decay such that ($p=2$):
\begin{align}
 \|\nabla u^k(t)\|_{L^2}\leq C\|w^k(t)\|_{L^2}\leq Ce^{-\frac{k}{4}t}.
\end{align}
Moreover, since $u_0\in \mathbb{S}^3 \in \dot{H}^{-1}$, combining Lemma \ref{priori estimate} with \eqref{cyz11}, one can easily get that
\begin{align}\label{see2}
\|u\|_{L^{\infty}_{t}(\dot{H}^{-1})} \leq C(\|u_0\|_{\dot{H}^{-1}}+\epsilon_0)\leq C.
\end{align}
By interpolation inequality, we obtain that
\begin{align}\label{4.00}
 \|u^k(t)\|^2_{L^2}\leq C\|u^k(t)\|_{\dot{H}^{-1}}\|w^k(t)\|_{L^2}\leq Ce^{-\frac{k}{4}t}.
\end{align}

On the other hand, in \eqref{th4-2} ,
since $u_0$  is axisymmetric, by \cite{ax} one can easily obtain a unique global solution $(u^0,\tau^0)$ with the same initial data $(u_0,\tau_0)(k)$. Although the coefficients of \eqref{th4-2} are independent of $k$, one can still look for the initial data which is dependent on $k$. Then, using the first equation (the classical Euler equation) of \eqref{th4-1}, we have
\begin{align}\label{4.0}
 \|u^0(t)\|_{L^2}= \|u_0\|_{L^2}= \epsilon_0.
\end{align}

Therefore, there exists a $T=\frac{1}{k^3}$ such that when $t\geq T$, we have
 \begin{align}\label{4.1}
 \|u^0(t)-u^k(t)\|_{L^{2}}&\geq \|u^0(t)\|_{L^{2}}-\|u^k(t)\|_{L^{2}}\notag\\
 &\geq \epsilon_0-Ce^{-\frac{k}{8}t} \notag\\
 &\geq \epsilon_0-Ck^2 \notag\\
 &\geq \epsilon_0-C\epsilon_0^2 \notag\\
 &= \frac{\epsilon_0}{2},
\end{align}
where the second inequality is based on \eqref{4.00}. This implies \eqref{non1} (~${B^{\frac{5}{2}}_{2,1}}\hookrightarrow L^2$) and completes the proof Theorem \ref{th4}.

\par\noindent
\section*{Acknowledgments}
This work is partially supported by the National Natural Science Foundation of China (Nos. 11801574, 11971485), Natural Science Foundation of Hunan Province (No. 2019JJ50788), Central South University Innovation-Driven Project for Young Scholars (No. 2019CX022) and Fundamental Research Funds for the Central Universities of Central South University, China (Nos. 2020zzts038, 2021zzts0041).


\end{document}